\documentclass[11pt]{article}

\usepackage{amsmath, amssymb, amsthm, graphicx, tikz-cd,xcolor,enumitem,natbib,hyperref}
\usepackage[margin=1in]{geometry}

\newcommand{\bZ}{\mathbb{Z}}

\newcommand\set[1]{\left\{ #1 \right\}}
\newcommand\abs[1]{\left\lvert #1 \right\rvert}
\newcommand\pars[1]{\left( #1 \right)}

\newcommand\mmod[1]{\ (\mathrm{mod}\ #1)}

\theoremstyle{definition}
\newtheorem{defn}{Definition}

\theoremstyle{plain}
\newtheorem{thm}{Theorem}

\newtheorem{lem}{Lemma}

\newtheorem{cor}{Corollary}

\title{A linear upper bound on the $\bZ_p$-Ramsey number of graphs with sufficiently large $2$-packing}
\author{Emily Heath \footnote{California State Polytechnic University Pomona, \texttt{eheath@cpp.edu}.} 
\and 
Andrew Simmons \footnote{California State Polytechnic University Pomona, \texttt{simmons1@cpp.edu}.}}
\date{\today}
\begin{document}
\maketitle

\begin{abstract}
Given a positive integer $k$ and graph $G$,  the $\mathbb{Z}_k$-Ramsey number $R(G,\bZ_k)$ is the least $N$ (if it exists) such that every coloring $f:E(K_N)\rightarrow \mathbb{Z}_k$ contains a copy $G'$ of $G$ such that $\sum_{e\in E(G')}f(e)=0$. 
Motivated by a question of Caro and Mifsud, we study the $\mathbb{Z}_k$-Ramsey number of graphs with a sufficiently large 2-packing, i.e. a set of vertices $S\subseteq V(G)$ such that $N[u]\cap N[v]=\emptyset$ for all distinct $u,v\in S$. In particular, we prove that $R(G,\mathbb{Z}_p)\leq n+6p-9$ for all $n$-vertex graphs $G$ and all primes $p$ such that $p$ divides $e(G)$, the minimum degree of $G$ is at least $1$, and there exists a $2$-packing of $G$ with size $p-1$. 
This upper bound improves depending on vertex degrees in the $2$-packing, with equality in certain cases. The result also implies an upper bound of the form $R(G,\mathbb{Z}_p)\leq n+C$ for $n$-vertex graphs $G$ of bounded maximum degree.

\end{abstract}

\section{Introduction}
The Erd\H{o}s-Ginzburg-Ziv theorem states that every sequence of $2k-1$ integers contains a subsequence of $k$ integers whose sum is congruent to $0$ modulo $k$. Bialostocki and Dierker~\cite{bialostocki1990zero} initiated the study of a graph-theoretic analogue of the theorem, establishing some of the first results in zero-sum Ramsey theory. While classical Ramsey theory involves coloring combinatorial structures and investigating the appearance of monochromatic substructures, zero-sum Ramsey theory involves interchanging colors with elements from a group (often $\bZ_k$) and instead focusing on zero-sum substructures.

More precisely, given a graph $G$ and positive integer $k$, the \textit{Ramsey number} $R(G,k)$ is the least $N$ for which every \textit{$k$-coloring} $f:E(K_N)\to \set{1,\ldots,k}$ contains a monochromatic copy of $G$; its existence is guaranteed by Ramsey's Theorem \cite{ramsey1987problem}. If $\Gamma$ is a finite abelian group, then the \textit{zero-sum Ramsey number} $R(G,\Gamma)$ is the least $N$ (if it exists) for which every \textit{$\Gamma$-coloring} $f:E(K_N)\to \Gamma$ contains a \textit{zero-sum copy} of $G$, i.e. a copy $G'$ of $G$ such that $\sum_{e\in E(G')}f(e)=0_\Gamma$. The \textit{exponent} of any finite group $\Gamma$, denoted $\exp(\Gamma)$, is the smallest positive integer $r$ such that $rg=0$ for all $g\in \Gamma$. Note that $R(G,\Gamma)$ exists if and only if $\exp(\Gamma)$ divides $e(G)$. If the divisibility condition holds, then monochromatic copies of $G$ are zero-sum, and therefore $R(G,\Gamma)\leq R\pars{G,\abs{\Gamma}}$. If the divisibility condition does not hold, then for all positive integers $N$, there exists a constant $\Gamma$-coloring of $K_N$ that fails to contain a zero-sum copy of $G$. In our work, we restrict our focus to the \textit{$\bZ_k$-Ramsey number} $R(G,\bZ_k)$, which exists if and only if $\exp(\bZ_k)=k$ divides $e(G)$.

Few exact values for the $\bZ_k$-Ramsey number are known, even for small values of $k$. For example, the $\bZ_2$-Ramsey numbers were completely determined in all cases by Caro \cite{CARO1994205}, but there are still cases that remain open for $k=3$. Recently, Alvarado, Colucci, and Parente \cite{alvarado2025problemcaromathbbz3ramseynumber} made significant progress by completely determining the $\bZ_3$-Ramsey number of forests. 

%In another direction, work has been focused on proving upper bounds for nice families of graphs but more general groups. In particular, 

However, it is known that the zero-sum Ramsey number is linear for a large family of graphs. In particular, Katz, Lian, Malekshahian, and Shapiro \cite{katz2025linearupperboundzerosum} showed that a linear upper bound of the form $R(G,\Gamma)\leq Cn$ holds for all $n$-vertex graphs of bounded maximum degree and finite abelian groups $\Gamma$ such that $\abs{\Gamma}$ divides $e(G)$.

In many cases for the $\bZ_k$-Ramsey number, the upper bound can be improved to the form  $R(G,\bZ_k)\leq n+C$. 
For example, Colucci and D'Emidio \cite{colucci2026linearupperboundzerosum} proved that $R(F,\bZ_p)\leq n+9p-12$ for all primes $p$ and all $n$-vertex forests $F$ with $p$ dividing $e(F)$, minimum degree at least $1$, and $n\geq 3p^2-12p+11$. Shapiro \cite{Shapiro} then generalized this result to $d$-degenerate graphs and gave a slight improvement in the case of forests (which are 1-degenerate graphs). A \textit{$d$-degenerate} graph is a graph in which every subgraph has at least one vertex of degree at most $d$. Shapiro \cite{Shapiro} showed that $R(G,\bZ_p)\leq n+(3+3d)p$ for all $d$-degenerate $n$-vertex graphs $G$ and all primes $p$ dividing $e(G)$ such that $2d<p\leq e(G)/(2d(d+1)^2)$. 

We now discuss our main result. Caro and Mifsud \cite{mod3} asked whether $R(G,\bZ_k)\leq n+k-1$ for all $n$-vertex graphs $G$ containing at least $k-1$ vertices of degree $1$ with pairwise distances of at least $3$. They proved this result for the case $k=3$. We answer their question in the affirmative for all primes and in fact prove a much more general result for graphs containing a $2$-packing of size $k-1$.  
A \textit{$2$-packing} of a graph $G$ is a subset $S\subseteq V(G)$ such that $N[u]\cap N[v]=\emptyset$ for all distinct $u,v\in S$, where $N[u]$ denotes the closed neighborhood $N(u)\cup\{u\}$.
\begin{thm}
    Let $p$ be prime. Suppose $G$ is an $n$-vertex graph such that $p$ divides $e(G)$, the minimum degree of $G$ is at least $1$, and there exists a $2$-packing $S$ of $G$ with $\abs{S}=p-1$. Let
    \begin{equation*}
        s_0=\abs{\set{v\in S:\deg(v)\equiv 0\mmod{p}}},
    \end{equation*}
    and
    \begin{equation*}
        c=\begin{cases}
            0 & \mbox{if $V(G)\neq \bigcup_{v\in S}N[v]$ or $\deg(v)=1$ for all $v\in S$,}\\
            1 & \mbox{otherwise.}\\
        \end{cases}
    \end{equation*}
    Then
    \begin{equation*}
        R(G,\bZ_p)\leq \begin{cases}
            n+p-1+c & \mbox{if $s_0=0$,}\\
            n+5p-9+s_0+c & \mbox{if $s_0\neq 0$.}\\
        \end{cases}
    \end{equation*}
    Moreover, equality holds in the case that $s_0=c=0$ if there exists an integer $d\not\equiv 0\mmod{p}$ such that $\deg(v)\equiv d\mmod{p}$ for all $v\in V(G)$.
    \label{thm:main}
\end{thm}
Note that in all cases of the hypothesis for Theorem \ref{thm:main}, we have $s_0+c\leq p$. Therefore the upper bound can be stated more simply as $R(G,\bZ_p)\leq n+6p-9$ for all $n$-vertex graphs $G$ such that $p$ divides $e(G)$, the minimum degree of $G$ is at least $1$, and there exists a $2$-packing of $G$ with size $p-1$. 

Moreover, we obtain an improved linear upper bound for the $\bZ_p$-Ramsey number of graphs of bounded maximum degree when $p$ is prime. Indeed, if $\Delta$ is a positive integer, then any $n$-vertex graph $G$ with maximum degree at most $\Delta$ contains a $2$-packing of size $\left\lfloor\frac{n}{1+\Delta+\Delta(\Delta-1)}\right\rfloor=\left\lfloor\frac{n}{\Delta^2+1}\right\rfloor$. Therefore Theorem \ref{thm:main} applies for such $G$ when $p$ divides $e(G)$, the minimum degree of $G$ is at least $1$, and $n\geq (p-1)(\Delta^2+1)$. Lastly, note the condition on minimum degree can now be dropped since $R(G\cup K_1,\Gamma)\leq R(G,\Gamma)+1$ for any graph $G$ and finite abelian group $\Gamma$ such that $\abs{\Gamma}$ divides $e(G)$.
\begin{cor}
    Suppose $p$ is prime and $\Delta$ is a positive integer. Then there exists a constant $C=C(p,\Delta)$ for which
    \begin{equation*}
        R(G,\bZ_p)\leq n+C
    \end{equation*}
    for all $n$-vertex graphs $G$ such that $p$ divides $e(G)$ and the maximum degree of $G$ is at most $\Delta$. In particular, it suffices to choose $C(p,\Delta)=\max\set{R(K_{(p-1)(\Delta^2+1)},p),\ 6p-9}$.
\end{cor}
The family of graphs for which Theorem \ref{thm:main} applies has substantial overlap with yet another well known family of graphs. Indeed, the equality case in the theorem determines the exact $\bZ_p$-Ramsey number of most cycles when $p$ is prime.
\begin{cor}
    If $p$ is prime and $n\geq 3p$ such that $p$ divides $n$, then $R(C_n,\bZ_p)=n+p-1$.
\end{cor}

During completion of this manuscript, we learned that Chi and He \cite{CH} independently obtained this same result for $R(C_n,\bZ_k)$ for all odd $k$ and $n\geq 35k^2$. More generally, they proved that $R(C_n,\bZ_k)\leq \max\{R(C_{2k},\bZ_k),n+k-1\}$ for all $k$. They also determined for wheel graphs $W_n=C_n+K_1$ that $R(W_{3k},\bZ_3)=3k+1$; note that Theorem \ref{thm:main} does not apply to such graphs. Lastly, they showed that $R(C_6,\bZ_3)=8$. It may be of interest that a modification of our arguments gives an alternative proof of this fact.

\section{Preliminaries}

We will use a generalized version of the Cauchy-Davenport theorem. 
\begin{thm}
    Let $A_1,\ldots,A_l$ be nonempty subsets of $\bZ_p$ for $p$ prime. Then
    \begin{equation*}
        \abs{A_1+\cdots+A_l}\geq \min\set{\sum_{i=1}^{l}\abs{A_i}-l+1,\ p}.
    \end{equation*}
    \label{thm:CD}
\end{thm}
In addition, we will use a weighted analogue of the Erd\H{o}s--Ginzburg--Ziv theorem proved by Grynkiewicz \cite{grynkiewicz}.
\begin{thm}
    Let $n,k\geq 2$ be positive integers. Suppose $a_1,\ldots,a_n$ is a sequence of integers such that $\sum_{i=1}^{n}a_i\equiv 0\mmod{k}$ and suppose $b_1,\ldots,b_{n+k-1}$ is another sequence of integers. Then there exists a rearranged subsequence $b_{j_1},\ldots,b_{j_n}$ such that $\sum_{i=1}^{n}b_{j_i}a_i\equiv 0\mmod{k}$.
    \label{thm:weightedEGZ}
\end{thm}
The following definition allows us to create a correspondence between special subgraphs of a $\bZ_p$-colored ``host" graph and additive sets, so that we may use the Cauchy--Davenport theorem.
\begin{defn}
    Let $p\geq 2$ be prime. Let $H$ be a complete graph and $f:E(H)\to \bZ_p$. Call the pair $(S,T)$ a \textit{switch} if $S,T\subseteq V(H)$ such that $\abs{S}\geq 2$, $\abs{T}\geq 1$, $S\cap T=\emptyset$, and $\abs{\set{\sum_{t\in T}f(st):s\in S}}=\abs{S}$. In other words, the stars with center in $S$ and $T$ as the set of leaves have distinct edge sums. We say that two switches $(S_1,T_1)$ and $(S_2,T_2)$ are \textit{disjoint} if $S_1$, $S_2$, $T_1$, and $T_2$ are pairwise disjoint.
    \label{defn:switch}
\end{defn}

\begin{figure}
    \centering
    \begin{tikzpicture}[scale=0.7, every node/.style={circle, draw, fill=black}, minimum size=1.5mm, inner sep=0pt]
        % outer
        \node (o0) at (0.8,2.9) {};
        \node (o1) at (2.1,2.1) {};
        \node (o2) at (2.9,0.8) {};
        \node (o3) at (2.9,-0.8) {};
        \node (o4) at (2.1,-2.1) {};
        \node (o5) at (0.8,-2.9) {};
        \node (o6) at (-0.8,-2.9) {};
        \node (o7) at (-2.1,-2.1) {};
        \node (o8) at (-2.9,-0.8) {};
        \node (o9) at (-2.9,0.8) {};
        \node (o10) at (-2.1,2.1) {};
        \node (o11) at (-0.8,2.9) {};
        \foreach \i in {0,...,10}{
            \draw (o\i) -- (o\the\numexpr\i+1\relax);
        }
        \draw (o11) -- (o0);
        % inner
        \node (i0) at (0.26,0.96) {};
        \node (i1) at (0.7,0.7) {};
        \node (i2) at (0.96,0.26) {};
        \node (i3) at (0.96,-0.26) {};
        \node (i4) at (0.7,-0.7) {};
        \node (i5) at (0.26,-0.96) {};
        \node (i6) at (-0.26,-0.96) {};
        \node (i7) at (-0.7,-0.7) {};
        \node (i8) at (-0.96,-0.26) {};
        \node (i9) at (-0.96,0.26) {};
        \node (i10) at (-0.7,0.7) {};
        \node (i11) at (-0.26,0.96) {};
        \foreach \i in {0,...,10}{
            \draw (i\i) -- (i\the\numexpr\i+1\relax);
        }
        % middle
        \draw (i11) -- (i0);
        \node (m1) at (1.4,1.4) {};
        \node (m2) at (1.92,0.52) {};
        \draw (m1) -- (m2);
        \node (m3) at (-1.4,1.4) {};
        \node (m4) at (-1.92,0.52) {};
        \draw (m3) -- (m4);
        \node (m5) at (0.52,-1.92) {};
        \node (m6) at (-0.52,-1.92) {};
        \draw (m5) -- (m6);
        % rest
        \foreach \i in {0,...,11}{
            \draw (o\i) -- (i\i);
        }
    \end{tikzpicture}
    \hspace{1cm}
    \begin{tikzpicture}[scale=0.7, every node/.style={circle, draw, fill=black}, minimum size=1.5mm, inner sep=0pt]
        % outer
        \node (o0) at (0.8,2.9) {};
        \node (o1) at (2.1,2.1) {};
        \node (o2) at (2.9,0.8) {};
        \node (o3) at (2.9,-0.8) {};
        \node (o4) at (2.1,-2.1) {};
        \node (o5) at (0.8,-2.9) {};
        \node (o6) at (-0.8,-2.9) {};
        \node (o7) at (-2.1,-2.1) {};
        \node (o8) at (-2.9,-0.8) {};
        \node (o9) at (-2.9,0.8) {};
        \node (o10) at (-2.1,2.1) {};
        \node (o11) at (-0.8,2.9) {};
        \foreach \i in {0,...,10}{
            \draw[gray] (o\i) -- (o\the\numexpr\i+1\relax);
        }
        \draw[gray] (o11) -- (o0);
        % inner
        \node (i0) at (0.26,0.96) {};
        \node (i1) at (0.7,0.7) {};
        \node (i2) at (0.96,0.26) {};
        \node (i3) at (0.96,-0.26) {};
        \node (i4) at (0.7,-0.7) {};
        \node (i5) at (0.26,-0.96) {};
        \node (i6) at (-0.26,-0.96) {};
        \node (i7) at (-0.7,-0.7) {};
        \node (i8) at (-0.96,-0.26) {};
        \node (i9) at (-0.96,0.26) {};
        \node (i10) at (-0.7,0.7) {};
        \node (i11) at (-0.26,0.96) {};
        \foreach \i in {0,...,10}{
            \draw[gray] (i\i) -- (i\the\numexpr\i+1\relax);
        }
        % middle
        \draw[gray] (i11) -- (i0);
        \node (m1) at (1.4,1.4) {};
        \node (m2) at (1.92,0.52) {};
        \draw[gray] (m1) -- (m2);
        \node (m3) at (-1.4,1.4) {};
        \node (m4) at (-1.92,0.52) {};
        % \draw (m3) -- (m4);
        \node (m5) at (0.52,-1.92) {};
        \node (m6) at (-0.52,-1.92) {};
        \draw[gray] (m5) -- (m6);
        % rest
        \foreach \i in {0,...,11}{
            \draw[gray] (o\i) -- (i\i);
        }
        % swaps
        \draw[ultra thick] (o9) -- (m4);
        \draw[ultra thick] (m3) -- (m4);
        \draw[ultra thick] (i9) -- (m4);
        \node (v1) at (-2.6,0.2) {};
        \draw[dashed] (v1) -- (o9);
        \draw[dashed] (v1) -- (m3);
        \draw[dashed] (v1) -- (i9);
        \draw[ultra thick] (o6) -- (o7);
        \draw[ultra thick] (o6) -- (o5);
        \draw[ultra thick] (o6) -- (m6);
        \node (v2) at (0,-2.6) {};
        \draw[dashed] (v2) -- (o7);
        \draw[dashed] (v2) -- (o5);
        \draw[dashed] (v2) -- (m6);
        \draw[ultra thick] (i11) -- (i0);
        \draw[ultra thick] (i0) -- (i1);
        \draw[ultra thick] (o0) -- (i0);
        \node (v3) at (0,2) {};
        \draw[dashed] (v3) -- (o0);
        \draw[dashed] (v3) -- (i11);
        \draw[dashed] (v3) -- (i1);
        \draw[ultra thick] (o2) -- (o3);
        \draw[ultra thick] (o3) -- (o4);
        \draw[ultra thick] (o3) -- (i3);
        \node (v4) at (2.5,0) {};
        \draw[dashed] (v4) -- (o2);
        \draw[dashed] (v4) -- (o4);
        \draw[dashed] (v4) -- (i3);
        \node (m1) at (1.4,1.4) {};
        \node (m2) at (1.92,0.52) {};
        \node (m3) at (-1.4,1.4) {};
        \node (m4) at (-1.92,0.52) {};
        \node (m5) at (0.52,-1.92) {};
        \node (m6) at (-0.52,-1.92) {};
    \end{tikzpicture}
    \caption{By Lemma \ref{lem:embed}, finding a zero-sum copy of the graph on the left in a $\bZ_5$-coloring reduces to finding a subgraph like the graph on the right. In every thick--dashed pair of stars, the stars must have distinct edge sums. The edge sum of the gray edges can be any value.}
    \label{fig:1}
\end{figure}
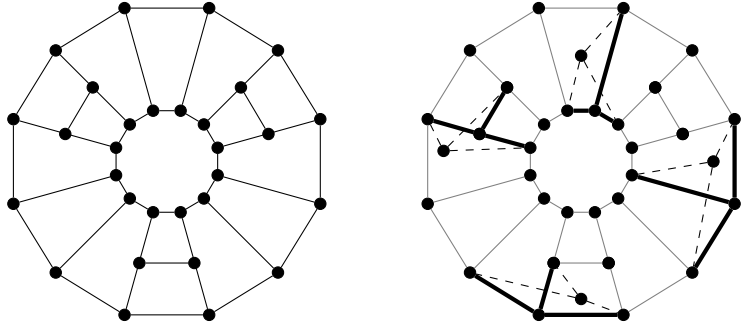

If we find an appropriate set of switches in a ``host" graph, then we can find copies of the ``target" graph with as many different edge sums as the number of elements in the sum set of the corresponding additive sets. In particular, if the corresponding sum set is the whole group $\bZ_p$, then one of the copies is guaranteed to be zero-sum. See Figure \ref{fig:1} for an example. The following lemma makes this precise.
\begin{lem}
    Let $p$ be prime. Suppose $G$ is an $n$-vertex graph with a $2$-packing $S=\set{v_1,\ldots,v_{l}}$. Let $H=K_{n+p-1}$ and $f:E(H)\to \bZ_p$. If $H$ contains a family of pairwise disjoint switches $\set{(S_i,T_i)}_{i=1}^{l}$ such that $\abs{T_i}=\deg(v_i)$ for all $i$ and $\sum_{i=1}^{l}\abs{S_i}=p+l-1$, then $H$ contains a zero-sum copy of $G$.
    \label{lem:embed}
\end{lem}
\begin{proof}
    Assume that $H$ contains such a family of switches $\set{(S_i,T_i)}_{i=1}^{l}$. For each $i$, let $A_i=\set{\sum_{t\in T_i}f(st):s\in S_i}$. We describe how to embed a zero-sum copy of $G$. For each $i$, map $N(v_i)$ injectively into $T_i$, but otherwise arbitrarily. Embed $G\setminus \pars{S\cup\bigcup_{i=1}^{l}N(v_i)}$ in $H\setminus\bigcup_{i=1}^{l} (S_i\cup T_i)$ arbitrarily; there are exactly enough vertices to do so. Let $c$ denote the edge sum of the resulting copy of $G\setminus S$. We may complete an embedding of $G$ by choosing for each $v_i\in S$ an image in $S_i$. Note that $|A_i|=|S_i|$ by definition of switches. Thus, Theorem \ref{thm:CD} yields
    \begin{equation*}
        \abs{A_1+\cdots + A_l}\geq \min\set{p,\ \sum_{i=1}^{l}\abs{S_i}-l+1}=p.
    \end{equation*}
    In particular, it is possible to choose images for the $v_i$ so that the resulting copy of the bipartite graph with parts $S$ and $\bigcup_{i=1}^lN(v_i)$ has edge sum $-c$. As a result, the total edge sum of the copy of $G$ is $0$, as desired.
\end{proof}

Monochromatic copies of the ``target" graphs that we consider (graphs with size divisible by the number of colors) are zero-sum. Hence, the next lemma will allow us to always assume that our $\bZ_p$-colored ``host" graphs contain small switches which are sufficient for embedding vertices of degree 1.
\begin{lem}
    Let $k\geq 2$ be an integer. Let $G$ be an $n$-vertex graph with $\delta(G)\geq 1$ which contains a $2$-packing $S=\set{v_1,\ldots,v_{k-1}}$. Then any $k$-coloring of $K_{n+k-1}$ contains a monochromatic copy of $G$ or at least $k-1$ disjoint, non-monochromatic copies of $P_3$.
    \label{lem:p3s}
\end{lem}
\begin{proof}
    The statement is trivial for $k=2$, so assume that $k\geq 3$. Let $H=K_{n+k-1}$ and let $f:E(H)\to \set{1,\ldots,k}$. Let $\set{a_ib_ic_i}_{i=1}^{l}$ be a collection of disjoint, non-monochromatic copies of $P_3$ in $H$ of maximum size. Assume towards a contradiction that $l\leq k-2$. Observe that $K=H\setminus\bigcup_{i=1}^{l}\set{a_i,b_i,c_i}$ is a monochromatic clique on $n+k-1-3l$ vertices; denote its color by $x$. We show that $K$ is contained in a larger monochromatic subgraph of $H$.
    
    Note that we may assume $f(a_ib_i)=x$ for all $i$. If instead both $f(a_ib_i)\neq x$ and $f(b_ic_i)\neq x$ for some $i$, then we may replace $a_ib_ic_i$ with another non-monochromatic path containing an edge of color $x$ as follows. Consider any distinct vertices $u_1,u_2,u_3\in V(K)$, which must exist since $n\geq 2(k-1)$ and therefore $\abs{V(K)}\geq 3$. 
    Since we must have either $f(a_ib_i)\neq f(b_iu_2)$ or $f(b_ic_i)\neq f(b_iu_2)$, we may assume without loss of generality that $f(a_ib_i)\neq f(b_iu_2)$. 
    It must be the case that at least one of $f(a_iu_1)$, $f(b_iu_2)$, or $f(c_iu_3)$ is equal to $x$, as otherwise it would be possible to obtain a larger collection of non-monochromatic paths by replacing $a_ib_ic_i$ with $a_ib_iu_2$ and $u_1u_3c_i$. Thus, we may replace the path $a_ib_ic_i$ with $u_1a_ib_i$, $u_2b_ia_i$, or $u_3c_ib_i$, respectively, increasing the number of paths in the collection with an edge of color $x$ while preserving the property that the clique outside this collection is still monochromatic in color $x$.

    Observe now that $f(a_iu)=x$ for all $i$ and all $u\in K$. Indeed, if $f(a_iu_1)\neq x$ for some $i$ and $u_1\in K$, then we may again find a larger collection of non-monochromatic paths by considering another pair of vertices $u_2,u_3\in V(K)\backslash\{u_1\}$. In particular, if $f(c_iu_3)=x$, replacing $a_ib_ic_i$ with $u_2u_1a_i$ and $u_3c_ib_i$ yields a larger collection of paths. If $f(c_iu_3)\neq x$, replacing $a_ib_ic_i$ with $b_ia_iu_1$ and $u_2u_3c_i$ yields the same contradiction. 

    Now, we may assume for all $i$ that either (i) $f(a_ic_i)=x$ and $f(c_iu)=x$ for all $u\in K$, or (ii) $f(b_iu)=x$ for all but possibly one vertex in $V(K)$. If not, then there is some $i$ such that both (i) and (ii) fail. If $f(c_iu)=x$ for all $u\in K$ but $f(a_ic_i)\neq x$, then we can replace $a_ib_ic_i$ in our collection of non-monochromatic paths with $b_ia_ic_i$, resulting in a path which satisfies (ii).  Otherwise, there exist distinct $u_1,u_2\in V(K)$ such that $f(c_iu_1)\neq x$ and $f(b_iu_2)\neq x$. A contradiction follows, since choosing $u_3\in V(K)$ distinct from $u_1,u_2$ and replacing $a_ib_ic_i$ with $a_ib_iu_2$ and $u_3u_1c_i$ yields a larger collection of paths. 

    Next, observe that $f(a_ia_j)=x$ for all $i\neq j$. If instead $f(a_ia_j)\neq x$ for some $i\neq j$, then we can find a larger collection of non-monochromatic paths chosen based on whether $a_ib_ic_i$ and $a_jb_jc_j$ satisfy (i) or (ii) above. If $a_ib_ic_i$ and $a_jb_jc_j$ both satisfy (i) above, then for any distinct $u_1,u_2,u_3\in V(K)$, we can use $u_1c_ib_i$, $u_2c_jb_j$, and  $u_3a_ia_j$. If $a_ib_ic_i$ satisfies (i) and $a_jb_jc_j$ satisfies (ii), then we can pick distinct $u_1,u_2,u_3\in V(K)$ so that $f(u_2b_j)=x$ and use $u_1c_ib_i$, $u_2b_jc_j$, and $u_3a_ia_j$. And if both satisfy (ii), then we can pick distinct $u_1,u_2,u_3\in V(K)$ so that $f(u_1b_i)=f(u_2b_j)=x$ and use $u_1b_ic_i$, $u_2b_jc_j$, and $u_3a_ia_j$.
    
    We now describe how to embed a monochromatic copy of $G$ using all of the $a_i$, one of $b_i$ or $c_i$ for each $i$, and all but one vertex in $V(K)$. Indeed, using at least two vertices from each path will be enough since $\abs{V(H)}-l\geq n+1$. Let $b_i'=c_i$ if $a_ib_ic_i$ satisfies (i) above. Otherwise let $b_i'=b_i$, and if it exists, let $u_i$ denote the vertex in $K$ such that $f(b_i'u_i)\neq x$. Note that the $u_i$ are not necessarily distinct. Let $v_1,\ldots,v_{k-1}$ denote the vertices in $S$. For $1\leq i\leq l$, map $v_i$ and one of its neighbors to $b_i'$ and $a_i$, respectively. While there exist $i\neq j$ such that $\deg(v_i),\deg(v_j)\geq 2$ and $u_i$ and $u_j$ exist and are distinct, map an unmapped neighbor of $v_i$ to $u_j$ and an unmapped neighbor of $v_j$ to $u_i$. Afterward, if there exists $i$ such that $\deg(v_i)\geq 2$, $u_i$ exists, and $u_i$ has not yet been assigned a preimage, map $v_{k-1}$ to $u_i$. Embed the rest of $G$ using vertices in $V(K)$, but otherwise arbitrarily. Because the $N[v_i]$ are disjoint, no edge of the form $b_i'u_i$ or $a_ib_j'$ or $b_i'b_j'$ for $i\neq j$ will be used. By our previous observations, the embedding of $G$ is monochromatic in color $x$.
\end{proof}
The next lemma allows us to compute edge sums easily for a specific family of $\bZ_k$-colorings when $k$ is odd.

\begin{lem}
    Let $k\geq 3$ be odd. Suppose $K$ is a $\bZ_k$-colored complete graph and there exists a partition of $V(K)$ into sets $V_1,\ldots,V_m$ such that $\abs{V_i}\geq 2$ for each $i$, each clique $K[V_i]$ is monochromatic in some color $c_i$, and the edges between cliques $K[V_i]$ and $K[V_j]$ are colored with $\overline{2}^{-1}(c_i+c_j)$. Then for any subgraph $G$ of $K$, its edge sum is given by $\overline{2}^{-1}\sum_{i=1}^{m}c_i\sum_{v\in V(G)\cap V_i}\deg_G(v)$.
    \label{lem:cliquecoloringsum}
\end{lem}
\begin{proof}
    For any subgraph $G$ of $K$, compute its edge sum as follows:
    \begin{align*}
        &\sum_{i=1}^{m}c_i\abs{E(G)\cap K[V_i]}+\sum_{i< j}\overline{2}^{-1}(c_i+c_j)\sum_{v\in V(G)\cap V_i}\abs{N_G(v)\cap V_j}\\
        &= \sum_{i=1}^{m}c_i\pars{\frac{1}{2}\sum_{v\in V(G)\cap V_i}\abs{N_G(v)\cap V_i}} + \sum_{i< j}\overline{2}^{-1}\pars{c_i\sum_{v\in V(G)\cap V_i}\abs{N_G(v)\cap V_j}+c_j\sum_{v\in V(G)\cap V_j}\abs{N_G(v)\cap V_i}}\\
        &= \overline{2}^{-1}\sum_{i=1}^{m}c_i\sum_{v\in V(G)\cap V_i}\deg_G(v).
    \end{align*}
\end{proof}
\section{Proof of Theorem \ref{thm:main}}
\begin{proof}
    We will begin by proving the upper bound.  Set $C=p-1+c$ if $s_0=0$ and $C=5p-9+s_0+c$ if $s_0\neq 0$. Let $H=K_{n+C}$ and consider any $\bZ_p$-coloring $f:E(H)\to \bZ_p$. 

    Note that if $\deg(v_i)=1$ for all $v\in S$, then the upper bound follows from Lemma \ref{lem:p3s} and Lemma \ref{lem:embed}, because together they imply that $H$ contains a zero-sum copy of $G$. Thus, we may assume we are not in this case.
    
    Suppose $S=\set{v_1,\ldots,v_{p-1}}$ such that $\deg(v_i)\not\equiv 0\mmod{p}$ for $1\leq i\leq p-1-s_0$. Say that a cycle $abcd$ in $H$ ``good" if $f(ab)-f(bc)\neq f(ad)-f(dc)$. 
    
    Assume first that $H$ contains pairwise disjoint non-monochromatic paths $\set{a_ib_ic_i}_{i=1}^{p-1-s_0}$ and good cycles $\set{a_ib_ic_id_i}_{i=p-s_0}^{p-1}$. (If $s_0=0$, then there are $p-1$ paths.) Choose pairwise disjoint subsets $\set{U_i}_{i=1}^{p-1}$ of $V(H)\setminus\pars{\bigcup_{i=1}^{p-1}\set{a_i,b_i,c_i}\cup\bigcup_{i=p-s_0}^{p-1}\set{d_i}}$ so that $\abs{U_i}=\deg(v_i)-1$ if $1\leq i\leq p-1-s_0$ and $\abs{U_i}=\deg(v_i)-2$ if $p-s_0\leq i\leq p-1$. Choose $u_0\in V(H)\setminus\pars{\bigcup_{i=1}^{p-1}\set{a_i,b_i,c_i}\cup\bigcup_{i=p-s_0}^{p-1}\set{d_i}}$ so that $u_0\notin U_i$ for all $i$. There are enough vertices to do so because
    \begin{align*}
        3(p-1-s_0)+4s_0+\sum_{i=1}^{p-1}\abs{U_i} &= 3(p-1)+\sum_{i=1}^{p-1}(\deg(v_i)-1)\\
        &=\sum_{i=1}^{p-1}(\deg(v_i)+1)+p-1\\
        &= \abs{\bigcup_{v\in S}N[v]}+p-1\\
        &< \abs{V(H)}.
    \end{align*}
    
    We describe how to find a family of pairwise disjoint switches $\set{(S_i,T_i)}_{i=1}^{p-1}$ such that $\abs{S_i}=2$ and $\abs{T_i}=\deg(v_i)$; then the upper bound will follow by Lemma \ref{lem:embed}. For each $i$ such that $\deg(v_i)=1$, let $S_i=\set{a_i,c_i}$ and $T_i=\set{b_i}$. We will consider two cases for $i$ such that $\deg(v_i)\geq 2$. See Figure \ref{fig:2} for an example.
    
    First, let $1\leq i\leq p-1-s_0$ (hence $\deg(v_i)\not\equiv 0\mmod{p}$ and $|U_i|=\deg(v_i)-1$). If $\sum_{u\in U_i\cup \set{u_0}}f(a_iu)\neq \sum_{u\in U_i\cup\set{u_0}}f(c_iu)$, then let $S_i=\set{a_i,c_i}$, $T_i=U_i\cup\set{u_0}$, and redefine $u_0=b_i$. Otherwise, choose $u'\in U_i\cup \set{u_0}$ such that
        \begin{equation*}
            \sum_{u\in \pars{U_i\cup\set{u_0}}\setminus\set{u'}}f(a_iu)+f(a_ib_i)\neq \sum_{u\in \pars{U_i\cup\set{u_0}}\setminus\set{u'}}f(c_iu)+f(c_ib_i),
        \end{equation*}
        and let $S_i=\set{a_i,c_i}$, $T_i=\pars{U_i\cup\set{u_0,b_i}}\setminus\set{u'}$, and redefine $u_0=u'$. At least one such $u'$ exists; if instead equality holds in the above for all $u'\in U_i\cup\set{u_0}$, then summing all of those equalities and adding $\sum_{u\in U_i\cup\set{u_0}}f(a_iu)=\sum_{u\in U_i\cup\set{u_0}}f(c_iu)$ to both sides of the resulting equation gives
        \begin{equation*}
            \deg(v_i)\sum_{u\in U_i\cup \set{u_0}}f(a_iu)+\deg(v_i)f(a_ib_i) = \deg(v_i)\sum_{u\in U_i\cup \set{u_0}}f(c_iu)+\deg(v_i)f(c_ib_i).
        \end{equation*}
        Then cancellations yield the contradiction $f(a_ib_i)=f(c_ib_i)$.
     
    Now let $p-s_0\leq i\leq p-1$  (hence $\deg(v_i)\equiv 0\mmod{p}$ and $|U_i|=\deg(v_i)-2$).  Then at least one of
        \begin{equation*}
            f(a_ib_i)+\sum_{u\in U_i\cup{\set{u_0}}}f(a_iu)\neq f(c_ib_i)+\sum_{u\in U_i\cup{\set{u_0}}}f(c_iu)
        \end{equation*}
        or
        \begin{equation*}
            f(a_id_i)+\sum_{u\in U_i\cup{\set{u_0}}}f(a_iu)\neq f(c_id_i)+\sum_{u\in U_i\cup{\set{u_0}}}f(c_iu)
        \end{equation*}
        is true, because both cannot be equalities while $f(a_ib_i)-f(c_ib_i)\neq f(a_id_i)-f(c_id_i)$. So, let $S_i=\set{a_i,c_i}$, $T_i=U_i\cup\set{u_0,b_i}$, and redefine $u_0=d_i$, or let $S_i=\set{a_i,c_i}$, $T_i=U_i\cup\set{u_0,d_i}$, and redefine $u_0=b_i$, depending on which of the above inequalities hold.
    \begin{figure}[h]
        \centering
        \begin{tikzpicture}[scale=1.2, every node/.style={circle, draw, fill=white}, minimum size=5mm, inner sep=0pt]
            \node (u0) at (0,-0.3) {$u_0$};
            
            % P3 part
            \draw (-2.8,0) ellipse [x radius=1.5, y radius=1];
            \node[draw=none] at (-2.8,1.5) {$U_1$};
            
            \node (a1) at (-3.5,-1.7) {$a_1$};
            \node (b1) at (-2.6,-2.5) {$b_1$};
            \node (c1) at (-2.3,-1.7) {$c_1$};
            
            \node[black, scale=0.5] (u11) at (-3.7,0.2) {};
            \node[black, scale=0.5] (u12) at (-3.2,0.5) {};
            \node[black, scale=0.5] (u13) at (-2.6,0.6) {};
            \node[black, scale=0.5] (u14) at (-2,0.4) {};
            
            \draw[ultra thick] (a1) -- (b1);
            \draw[dashed] (b1) -- (c1);
            \draw[ultra thick] (a1) -- (u11);
            \draw[ultra thick] (a1) -- (u12);
            \draw[ultra thick] (a1) -- (u13);
            \draw[ultra thick] (a1) -- (u14);
            \draw[dashed] (u11) -- (c1);
            \draw[dashed] (u12) -- (c1);
            \draw[dashed] (u13) -- (c1);
            \draw[dashed] (u14) -- (c1);
            \draw[gray] (a1) -- (u0);
            \draw[gray] (c1) -- (u0);
            
            %C4 part
            \draw (2.8,0) ellipse [x radius=1.5, y radius=1];
            \node[draw=none] at (2.8,1.5) {$U_{2}$};
            
            \node (a2) at (2.3,-1.7) {$a_{2}$};
            \node (d2) at (3.2,-2.6) {$d_{2}$};
            \node (c2) at (3.5,-1.7) {$c_{2}$};
            \node (b2) at (2,-2.6) {$b_{2}$};
            
            \draw[ultra thick] (a2) -- (d2);
            \draw[dashed] (d2) -- (c2);
            
            \node[black, scale=0.5] (u21) at (3.7,0.2) {};
            \node[black, scale=0.5] (u22) at (3.2,0.5) {};
            \node[black, scale=0.5] (u23) at (2.6,0.6) {};
            \node[black, scale=0.5] (u24) at (2,0.4) {};
            
            \draw[gray] (a2) -- (b2);
            \draw[gray] (b2) -- (c2);
            \draw[ultra thick] (a2) -- (u21);
            \draw[ultra thick] (a2) -- (u22);
            \draw[ultra thick] (a2) -- (u23);
            \draw[dashed] (u21) -- (c2);
            \draw[dashed] (u22) -- (c2);
            \draw[dashed] (u23) -- (c2);
            \draw[ultra thick] (a2) -- (u0);
            \draw[dashed] (c2) -- (u0);
            \draw[dashed] (c2) -- (u24);
            \draw[ultra thick] (a2) -- (u24);
        \end{tikzpicture}
        \caption{An example of switches that may arise in the proof of Theorem \ref{thm:main} for $p=3$ and $s_0=c=1$. In every thick--dashed pair of stars, the stars have distinct edge sums. The gray edges are not included in a switch.}
        \label{fig:2}
    \end{figure}
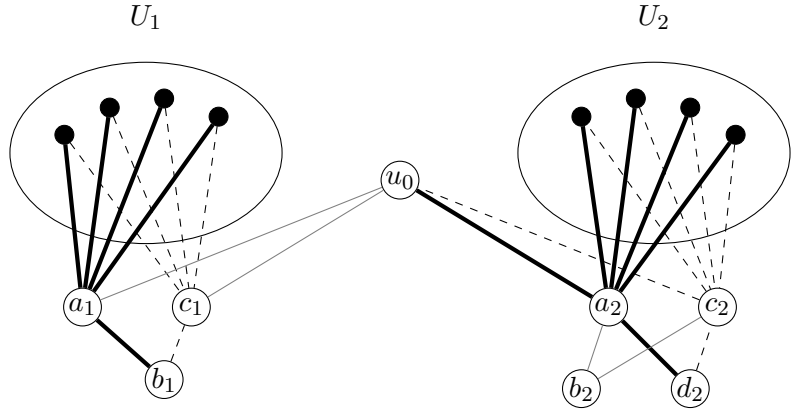

    Recall that by Lemma \ref{lem:p3s}, we may assume there are enough non-monochromatic paths. Above we showed if there are enough good cycles as well, then we can find the necessary switches and be done by Lemma \ref{lem:embed}. Thus, we now assume we do not have the sufficient number $s_0$ of good cycles for this approach.
    
    So, it only remains to prove the result when $s_0\geq 1$ and the number of good cycles is at most $s_0-1$. In this case, there exists a clique $K\subseteq H$ such that \[\abs{V(K)}\geq n+C-3(p-1-s_0)-4(s_0-1)\geq n+2p-2\] and $K$ does not contain any good cycles. Our goal is to show there exists a subgraph of $K$ for which the $\bZ_p$-coloring is as described in the hypothesis of Lemma \ref{lem:cliquecoloringsum}.
    
    First, choose any $v\in V(K)$; its incident edges in $K$ are colored with $m\leq p$ distinct colors. 
    Partition $V(K)\setminus\{v\}$ into $V_1,\ldots, V_m$ so that $f(vu_1)=f(vu_2)$ for $u_1\in V_i$ and $u_2\in V_j$ if and only if $i=j$. We will assume $|V_i|\geq |V_j|$ for all $i<j$.  Note that each clique $K[V_i]$ is either a single vertex or monochromatic in some color $c_i$, and  the edges between cliques $K[V_i]$ and $K[V_j]$ are colored with a single color $c_{ij}$, since there can be no good cycles in $K$.    
    
    Now, update $K$ by removing $v$ and any singleton $V_i$, and assume the remaining sets are $V_1,\ldots,V_{m'}$. If $p-1$ of the $V_i$ are singletons, then $|V_1|\geq n+2p-3-(p-1)=n+p-2\geq n$, guaranteeing the existence of a monochromatic copy of $G$ in $K[V_1]$. Thus, we may assume at most $p-2$ of the $V_i$ are singletons and $\abs{V(K)}\geq n+2p-3-(p-2)=n+p-1$. 

    We address the case $p=2$ first. Note that it cannot be the case that $m=2$ and $c_1\neq c_2$, since any copy of $C_4$ with three edges in one color and one edge in another color must be good. Moreover, we are done if $m=1$ or both $m=2$ and $c_1=c_2=c_{12}$, since then any copy of $G$ in $K$ is monochromatic and therefore zero-sum. Thus the only case that remains to be checked is $m=2$ and $c_1=c_2\neq c_{12}$. We may assume at this point that all vertex degrees in $G$ are even. If not, we could have found a $2$-packing $S$ of $G$ of size $2-1=1$ such that $s_0=0$, which is a case we already resolved. Hence every cut of $G$ is of even size. But embeddings of $G$ in $K$ correspond to cuts of $G$, and since $e(G)$ is even, it follows that they are all zero-sum.
    
    Now assume $p\neq 2$. We finally argue, for all $i\neq j$, that $c_{ij}=\overline{2}^{-1}(c_i+c_j)$. Choose distinct $u_1,u_2\in V_i$ and distinct $u_1',u_2'\in V_j$. Since $u_1u_2u_2'u_1'$ cannot be a good cycle, we must have $f(u_1u_2)-f(u_1u'_1)=f(u_2u_2')-f(u_2'u_1')$. Note that $f(u_1u_2)=c_i$ and $f(u'_1u'_2)=c_j$ and $f(u_1u'_1)=f(u_2u'_2)=c_{ij}$. Substitution yields $c_i-c_{ij}=c_{ij}-c_j$, and solving for $c_{ij}$ yields the claim. 
    
    Let $V(G)=\set{a_1,\ldots,a_n}$. By Lemma \ref{lem:cliquecoloringsum}, $K$ contains a zero-sum copy of $G$ if there exists a rearranged subsequence $b_{j_1},\ldots,b_{j_{n}}$ of $\underbrace{c_1,\dots,c_1}_{\abs{V_1} \text{ times}},\underbrace{c_2,\dots,c_2}_{\abs{V_2} \text{ times}},\dots,\underbrace{c_{m'},\dots,c_{m'}}_{\abs{V_{m'}} \text{ times}}$ such that $\sum_{t=1}^{n}b_{j_t}\deg(a_t)\equiv 0\mmod{p}$. But that is guaranteed by Theorem \ref{thm:weightedEGZ}, since $\sum_{t=1}^{n}\deg(a_t)=2e(G)\equiv 0\mmod{p}$.

%%%%%%%%%%%%%%%%%%%%%%%%%
    We conclude the proof by justifying that a matching lower bound holds when $s_0=c=0$, i.e. $C=p-1$, if there exists $d\not\equiv 0\mmod{p}$ such that $\deg(v)\equiv d\mmod{p}$ for all $v\in V(G)$. For $p=2$, consider the coloring of $K_{n}$ that sends all edges of a copy of $K_{1,n-1}$ to $\overline{1}$ and all other edges to $\overline{0}$. In this case, all vertex degrees in $G$ are assumed to be odd, so no copy of $G$ in $K_n$ is zero-sum. For $p\geq 3$, let $H=K_{n+p-2}$ and let $V_1$ and $V_2$ partition $V(H)$ such that $\abs{V_1}=n-1$ and $\abs{V_2}=p-1$. Color the edges of $H[V_1]$ with $\overline{0}$, the edges of $H[V_2]$ with $\overline{2}$, and the edges between the two cliques with $\overline{1}$. By Lemma \ref{lem:cliquecoloringsum}, the edge sum of any copy $G'$ of $G$ in $H$ is given by $\abs{V(G')\cap V_2}\overline{d}$, which is nonzero because $1\leq \abs{V(G')\cap V_2}\leq p-1$.
\end{proof}

\bibliography{citations}

@article{CH,
      title={On zero-sum {R}amsey numbers of cycles and wheels}, 
      author={Cheng Chi and Jialin He},
      year={2026},
      journal={arXiv:2605.14954}, 
}

@article{shapiro,
      title={A linear upper bound on zero-sum {R}amsey numbers of $d$-degenerate graphs in $\mathbb{Z}_p$}, 
      author={Andrey Shapiro},
      year={2026},
      journal={arXiv:2604.10864}, 
}

@article{article,
author = {Caro, Yair},
year = {1992},
month = {06},
pages = {},
title = {On zero-sum {R}amsey numbers--complete graphs},
volume = {2},
journal = {The Quarterly Journal of Mathematics},
doi = {10.1093/qmath/43.2.175}
}

@article{bialostocki1990zero,
  title={ON ZERO SUM {R}AMSEY NUMBERS--SMALL GRAPHS},
  author={Bialostocki, A and Dierker, P},
  journal={Ars Combinatoria},
  volume={29},
  pages={193--198},
  year={1990},
  publisher={CHARLES BABBAGE RES CTR PO BOX 272 ST NORBERT POSTAL STATION, WINNIPEG MB~…}
}

@incollection{ramsey1987problem,
  title = {On a problem of formal logic},
  author = {Ramsey, Frank P},
  booktitle = {Classic Papers in Combinatorics},
  pages = {1--24},
  year = {1987},
  publisher = {Springer}
}

@article{CARO1994205,
title = {A complete characterization of the zero-sum (mod 2) {R}amsey numbers},
journal = {Journal of Combinatorial Theory, Series A},
volume = {68},
number = {1},
pages = {205-211},
year = {1994},
issn = {0097-3165},
doi = {https://doi.org/10.1016/0097-3165(94)90098-1},
url = {https://www.sciencedirect.com/science/article/pii/0097316594900981},
author = {Yair Caro}
}

@article{alvarado2025problemcaromathbbz3ramseynumber,
      title={On a problem of {C}aro on $\mathbb{Z}_3$-{R}amsey number of forests}, 
      author={José D. Alvarado and Lucas Colucci and Roberto Parente},
      year={2025},
      journal={arXiv:2503.01032}
}

@article{mod3,
  title={On zero-sum {R}amsey numbers modulo 3},
  author={Caro, Yair and Mifsud, Xandru},
  journal={arXiv:2502.03864},
  year={2025}
}

@article{colucci2026linearupperboundzerosum,
      title={A linear upper bound on the zero-sum {R}amsey number of forests in $\mathbb{Z}_p$}, 
      author={Lucas Colucci and Marco D'Emidio},
      year={2026},
      journal={arXiv:2512.06229}, 
}

@article{katz2025linearupperboundzerosum,
      title={A linear upper bound for zero-sum {R}amsey numbers of bounded degree graphs}, 
      author={Jasmin Katz and Xiaopan Lian and Alexandru Malekshahian and Andrey Shapiro},
      year={2025},
      journal={arXiv:2512.17790}, 
}

@article{grynkiewicz,
author = {Grynkiewicz, David},
year = {2006},
month = {08},
pages = {445-453},
title = {A Weighted {E}rdős-{G}inzburg-{Z}iv Theorem},
volume = {26},
journal = {Combinatorica},
doi = {10.1007/s00493-006-0025-y}
}
\bibliographystyle{abbrv}
\end{document}